\documentclass[11pt]{amsart}
\usepackage{graphicx}
\usepackage[linktocpage=true,colorlinks,citecolor=blue,linkcolor=blue,urlcolor=blue]{hyperref}
\usepackage{amssymb}
\usepackage{cite}
\usepackage{amsmath}
\usepackage{latexsym}
\usepackage{amscd}
\usepackage{tikz}
\usepackage{amsthm}
\usepackage{mathrsfs}
\usepackage{url}
\usepackage[utf8]{inputenc}
\usepackage[english]{babel}
\usepackage{amsfonts}

\numberwithin{equation}{section}
\def\inv{\iota}
\def\binv{\overline{\iota}}
\def\R{\mathbb R}
\def\Z{\mathbb Z}

\def\P{\mathbb P}

\def\Fp{\mathbb F_p}

\def\ee{\varepsilon}

\DeclareMathOperator\GL{GL}

\DeclareMathOperator\SL{SL}

\newtheorem{theorem}{Theorem}[section]
\newtheorem{lemma}[theorem]{Lemma}
\newtheorem{proposition}[theorem]{Proposition}

\newtheorem{corollary}[theorem]{Corollary}

\theoremstyle{remark}
\newtheorem{remark}[theorem]{Remark}

\theoremstyle{definition}

\theoremstyle{remark}

\numberwithin{equation}{section}

\begin{document}
\title{Mixing time of fractional random walk on finite fields }
\author{Jimmy He}
\address{Department of Mathematics, Stanford University, Stanford, CA, USA}
\email{jimmyhe@stanford.edu}
\author{Huy Tuan Pham}
\address{Department of Mathematics, Stanford University, Stanford, CA, USA}
\email{huypham@stanford.edu}
\author{Max Wenqiang Xu}
\address{Department of Mathematics, Stanford University, Stanford, CA, USA}
\email{maxxu@stanford.edu}
\maketitle
\begin{abstract}
    We study a random walk on $\mathbb{F}_p$ defined by $X_{n+1}=1/X_n+\varepsilon_{n+1}$ if $X_n\neq 0$, and $X_{n+1}=\varepsilon_{n+1}$ if $X_n=0$, where $\varepsilon_{n+1}$ are independent and identically distributed. This can be seen as a non-linear analogue of the Chung--Diaconis--Graham process. We show that the mixing time is of order $\log p$, answering a question of Chatterjee and Diaconis \cite{CD20}.
\end{abstract}
\section{Introduction}
This paper studies a non-linear analogue of the Chung--Diaconis--Graham process, defined on $\Fp$ by
\begin{equation*}
    X_{n+1}=aX_n+\ee_{n+1},
\end{equation*}
where $a\in\Fp^\times$ is fixed and the $\varepsilon_i$ are independent and identically distributed. The mixing time of this Markov chain has been extensively studied, and it is now known that for certain $a$ (such as $a=2$) and almost all $p$, cutoff occurs at time $c\log p$ for an explicit constant $c$ \cite{EV}. 

Since simple random walk on $\Fp$ requires order $p^2$ steps to mix (see \cite{LP17} for example), the Chung--Diaconis--Graham process provides an explicit example of a random walk where applying a deterministic bijection (in this case, $x\mapsto ax$) between steps of the walk exponentially speeds up mixing. This was studied in \cite{CD20}, where it was asked whether other explicit examples could be provided.

In this paper, we consider the random walk on $\Fp$ defined by
\begin{equation*}
    X_{n+1}=\iota(X_n)+\ee_{n+1},
\end{equation*}
where $\iota(x)=1/x$ if $x\neq 0$ and $\iota(0)=0$, and the $\varepsilon_i$ are independent and identically distributed. Chatterjee and Diaconis asked about the order of the mixing time for this walk \cite{CD20}, which was originally suggested by Soundararajan. We solve this problem by showing that this random walk mixes in order $\log p$ steps. Since simple random walk on $\Fp$ mixes in order $p^2$ steps, this gives an exponential speedup for the mixing time, and provides an explicit example of when adding a deterministic function between steps of a Markov chain gives an exponential speedup. 

This model is similar in spirit to random walk on $\Fp$ given by moving from $x$ to one of $x+1$, $x-1$ or $x^{-1}$, which is known to be an expander (see Theorem 8.8 of \cite{W19} for example). Our results are not directly implied by this, but similar ideas underlie the results for both models. We discuss the relationship with this model in more detail in Remark \ref{rmk: symm model}.

The proof uses comparison theory developed by Smith \cite{S15} to relate the random walk on $\Fp$ to a random walk on the projective line $\P^1(\Fp)$. The walk on $\P^1(\Fp)$ is the quotient of a random walk on $\SL_2(\Fp)$ which is known to have a constant order spectral gap by results of Bourgain and Gamburd \cite{BG}. While our methods identify the order of the mixing time, they are not strong enough to identify the constant or obtain cutoff, and we leave this as an open problem.

Our result also implies bounds on the number of solutions to the congruence $xy\equiv 1\pmod{p}$ for $x\in I$ and $y\in J$ intervals of the same length, although these bounds are weaker than the ones established in \cite{CG11}.

\subsection{Main results}
We now formally state our main results. For two probability measures $\mu$ and $\nu$, let
\begin{equation*}
    \|\mu-\nu\|_{TV}=\sup_{A}|\mu(A)-\nu(A)|
\end{equation*}
denote the \emph{total variation distance}. Let $X_n$ be an ergodic Markov chain with transition matrix $P$ and stationary distribution $\pi$, . Let $P^n(x,y)$ denote the probability that $X_n=y$ given that $X_0=x$. Then the \emph{mixing time} is defined to be
\begin{equation*}
    t_{mix}(\ee)=\inf\left\{n\colon \sup _{x}\|P^n(x,\cdot)-\pi\|_{TV}\leq \ee\right\}.
\end{equation*}

Let $\inv:\Fp\to \Fp$ be the  function defined by $\inv(x)=1/x$ if $x\ne 0$ and $\inv(0)=0$. Let $\ee_i$ be independent and identically distributed on $\Z$. Let $K$ denote the transition matrix for the Markov chain on $\Fp$ defined by
\begin{equation}
\label{eq: walk defn}
    X_{n+1} =  \inv(X_n) + \ee_{n+1}.
\end{equation}
We show that for this Markov chain, $t_{mix}(\ee)$ is $\Theta(\log p)$, where the implied constants are allowed to depend on both $\mu$ and $\ee$. Informally, this means that order $\log p$ steps are necessary and sufficient for the Markov chain to converge to its stationary distribution. The following theorems give the desired lower and upper bounds on $t_{mix}$ respectively
.
\begin{theorem}[Lower bound]
\label{thm: lower bound}
Let $\mu$ denote a probability measure on $\Z$, and assume that
\begin{equation*}
    H(\mu)=\sum -\mu(x)\log \mu(x)<\infty.
\end{equation*}
Let $K$ be the transition matrix for the Markov chain on $\Fp$ defined by
\begin{equation*}
    X_{n+1}=\inv(X_n)+\ee_{n+1},
\end{equation*}
where the $\ee_n$ are independent and identically distributed according to $\mu$. Let $\pi$ denote the uniform distribution on $\Fp$. Then for all $x$ in $ \Fp$ 
\begin{equation*}
    \|K^n(x,\cdot)-\pi\|_{TV}\geq 1-\frac{nH(\mu)+\log 2}{\log p}.
\end{equation*}
\end{theorem}

\begin{theorem}[Upper bound]
\label{thm: upper bound}
Let $\mu$ denote a probability measure on $\Z$ whose support contains at least two values. Let $K$ be the transition matrix for the Markov chain on $\Fp$ defined by
\begin{equation*}
    X_{n+1}=\inv(X_n)+\ee_{n+1},
\end{equation*}
where the $\ee_n$ are independent and identically distributed according to $\mu$. Let $\pi$ denote the uniform distribution on $\Fp$. Then there exists a constant $C>0$ depending only on $\mu$ such that for all $x$ in $ \Fp$ and sufficiently large $p$,
\begin{equation*}
    \|K^n(x,\cdot)-\pi\|_{TV}\leq \frac{\sqrt{p}}{2}e^{-Cn}.
\end{equation*}
\end{theorem}
The lower bound follows easily from entropy considerations. We make an assumption that the entropy of $\mu$ is finite, which cannot be removed entirely (see Remark \ref{rmk: necc of finite entropy}). However, this assumption is relatively weak, and indeed all finitely supported distributions have finite entropy.

We now describe our plan to prove the upper bound. We establish the upper bound by connecting the Markov chain $X_n$ with a Markov chain on $\P^{1}(\Fp)$, which in turn is a projection of a Markov chain on a Cayley graph of $\SL_2(\Fp)$. We next sketch this connection. 

Recall that given a group $G$ and a generating set $S$, we define the \emph{Cayley graph} $\mathcal{G}(G, S)$ as the graph with vertex set $G$ and two elements $x$ and $y$ are connected if and only if $x= \sigma y$ for some $\sigma \in S$. We use $A(\mathcal{G})$ to denote the normalised adjacency matrix of $\mathcal{G}$, whose largest eigenvalue is 1. Let $S$ be a subset of $\SL_2(\Z)$ so that $S_p$ generates $\SL_2(\Fp)$, where $S_p$ is the set $S$ mod $p$. Results of Bourgain and Gamburd \cite{BG} give conditions for when $\mathcal{G}(\SL_2(\Fp),S_p)$ has a constant spectral gap, i.e., $\lambda_2(A(\mathcal{G}(\SL_2(\Fp),S_p)))\le 1-c$ for some constant $c>0$ independent of $p$, and thus also a $O(\log p)$ mixing time.

The random walk on $\mathcal{G}(\SL_2(\Fp),S_p)$ has a natural projection on $\P^{1}(\Fp)$, defined by the action 
\[ \begin{pmatrix}
a & b  \\
c & d \\
\end{pmatrix} \cdot x = \frac{ax+b}{cx+d}. \]
This can also be seen as a random walk on the Schreier graph corresponding to this action. Since the spectral gaps of the walks on $\SL_2(\Fp)$ and $\P^1(\Fp)$ are related, this gives an $O(\log p)$ mixing time for the walk on $\P^1(\Fp).$



Finally, we connect the Markov chain on $\P^{1}(\Fp)$ with the fractional Markov chain $X_n$ on $\Fp$. With an appropriate choice of the set $S$, the Markov chain on $\P^{1}(\Fp)$ induced by the random walk on $\mathcal{G}(\SL_2(\Fp),S_p)$ and $X_n$ have very similar transition probabilities. The mixing time bound of $X_n$ can thus be obtained from that of the chain on $\P^{1}(\Fp)$ via comparison theory of Markov chains on different state spaces. 


A key input to the above plan is the spectral gap of random walks on Cayley graphs of $\SL_2(\Fp)$, which we discuss in further detail in Section \ref{sec: rw on cayley graphs}.

\begin{remark}
\label{rmk: symm model}
The random walk on $\Fp$ obtained by moving from $x$ to $x+1$, $x-1$ or $x^{-1}$ with equal probability has been previously studied. In Section 3.3 of \cite{S90}, the walk on $\P^1(\Fp)$ is shown to have a constant order spectral gap, and Theorem 8.8 of \cite{W19} states that the same is true on $\Fp$. This model can be seen as an explicit version of the speedup proposed in \cite{HSS20}, who study random walks on graphs with an additional random perfect matching.

The proofs in both models also follow similar ideas, although the random walk we study is more challenging due to the lack of reversibility. In particular, our results imply that the walks above have a constant spectral gap, while the reverse is not true. Theorem 8.8 of \cite{W19} is stated without proof, and we could not find an explanation for the reduction between the study of the random walk on $\Fp$ to $\P^1(\Fp)$ in the literature, and so we note that the obvious analogue of Proposition \ref{prop: comp} in this setting would give a proof.
\end{remark}

\subsection{Bounds on solutions to a congruence equation}
We also give an application of these ideas to bounding the number of solutions to the congruence
\begin{equation*}
    xy\equiv 1\pmod{p}
\end{equation*}
with $x\in I$ and $y\in J$ for intervals $I$ and $J$ of the same length $m\leq p/2$. In particular, for large enough $p$ and $m$, we establish that
\begin{equation*}
    |\{(x,y)\in I\times J\mid xy\equiv 1\pmod{p}\}|\leq (1-\delta)m
\end{equation*}
for some absolute constant $\delta>0$ (see Theorem \ref{thm: application}). The proof is simple given the spectral gap estimates in the proof of Theorem \ref{thm: upper bound}. 

Our bound on the number of solutions to the congruence $xy\equiv 1\pmod{p}$ is weaker than the bounds obtained in \cite{CS10} and \cite{CG11} for intervals that are not too large and weaker than the standard ones coming from estimates on incomplete Kloosterman sums for intervals that are not too small. The number of solutions has also been estimated in \cite{G06, CZ17, Shk}. In particular, the method used in \cite{Shk} also relies on $\SL_2$ action.  While this result is not new, the proof is straightforward and we hope that our ideas can lead to further applications. 

\subsection{Related work}
The random walk we study may be viewed as a non-linear version of the Chung--Diaconis--Graham process, or the $ax+b$ process, which is the random walk on $\Fp$ (or more generally $\Z/n\Z$ for composite $n$) defined by
\begin{equation*}
    X_{n+1}=aX_n+\ee_{n+1},
\end{equation*}
where $a\in\Fp^\times$ is fixed and the $\ee_i$ are independent and identically distributed, drawn from some distribution. It was introduced in \cite{CDG}, where the case of $a=2$ and $\ee_i$ distributed uniformly on $\{-1,0,1\}$ was studied in detail. They showed that the mixing time was at most order $\log p\log\log p$, and that for almost all odd $p$, order $\log p$ steps were sufficient, although for infinitely many odd $p$,  order $\log p\log \log p$ steps were necessary. This process was subsequently studied in \cite{H06, H09, N11, H19, BV19, EV}. In particular, Eberhard and Varj\'u recently showed in \cite{EV} that for almost all $p$, the walk exhibits cutoff at $c\log p$ for some explicit constant $c\approx 1.01136$.

The Chung--Diaconis--Graham process gives an example of a speedup phenomenon that occurs when applying bijections between steps of a Markov chain. This was studied in \cite{CD20}, where it was shown that for a Markov chain on a state space with $n$ elements, almost all bijections, when applied between steps of a Markov chain, cause the chain to mix in $O(\log n)$ steps. Our work gives another example of this phenomenon, which fits into a broader theme of studying how convergence can be sped up, via either deterministic or random means \cite{HSS20, BQZ20, ABLS07}.

In \cite{He20}, the first author studied a more general non-linear version of the Chung--Diaconis--Graham process, defined by
\begin{equation*}
    X_{n+1}=f(X_n)+\ee_{n+1},
\end{equation*}
where $f$ is a bijection on $\Fp$. It was shown that for functions $f(x)$ which were extensions of rational functions, the mixing time was of order at most $p^{1+\ee}$ for any $\ee>0$, where the implicit constant depends only on $\ee$ and the degree of the polynomials appearing in $f(x)$. In particular, this applies for the function $\iota$ that we consider. However, the only known lower bound on the mixing time is of order $\log p$. Our work closes the large gap between the upper and lower bounds in this special case.

\subsection{Random walk on Cayley graphs of \texorpdfstring{$\SL_2(\Fp)$}{SL2(Fp)}}
\label{sec: rw on cayley graphs}

Selberg's Theorem \cite{Selberg} shows that if $S$ is a subset of $\SL_2(\mathbb{Z})$ such that $S$ generates a subgroup of $\SL_2(\mathbb{Z})$ of finite index, then $$\lim \sup_{p\to \infty} \lambda_2(A(\mathcal{G}(\SL_2(\Fp),S_p)))<1.$$ Bourgain and Gamburd \cite{BG} strengthen this result and show that the above property holds if and only if $S$ generates a non-elementary subgroup of $\SL_2(\mathbb{Z})$. 
Weigel \cite{W96} gives a convenient characterization that $\langle S\rangle$ is non-elementary if and only if $\langle S_p\rangle = \SL_2(\Fp)$ for some prime $p\ge 5$ (see also \cite{R10}). 

We remark that a simple reduction allows us to deduce Theorem \ref{thm: upper bound} from the special case where the distribution of $\ee_n$ is supported on two points. In this case, we can construct the set $S$ corresponding to the fractional Markov chain and easily verify that $\langle S\rangle$ is non-elementary. 

Golumbev and Kamper \cite{GK} show that under more stringent assumptions on the set of generators $S$, the Cayley graphs $\mathcal{G}(\SL_2(\Fp),S_p)$ have a stronger property that implies cut-off for the associated projected random walk on $\P^{1}(\Fp)$. This property also holds with high probability for random generating sets $S_p$. We also remark that the result of Bourgain-Gamburd has been further developed in various aspects, e.g., \cite{BV, BreuillardGamburd, GV, GK}. These results also have many applications in number theory, including sum product problems on finite fields  \cite{BGS, Helf08}. 

\subsection{Outline}
In Section \ref{sec: lower bound}, we prove Theorem \ref{thm: lower bound}. In Section \ref{sec: comparison}, we explain the comparison theory for Markov chains used to relate the spectral gaps for the walks on $\Fp$ and $\P^1(\Fp)$. In Section \ref{sec: proof}, we give a proof of Theorem \ref{thm: upper bound}. Finally, in Section \ref{sec: application}, we give a bound for the number of solutions to $xy\equiv 1\pmod{p}$ in a square $I_1\times I_2\subseteq \mathbb{F}_p^2$.

\subsection{Notations}
Throughout, we let $\inv:\Fp\to \Fp$ be the function defined by $\inv(x)=1/x$ if $x\ne 0$ and $\inv(x)=0$ if $x=0$. We view $\P^1(\Fp)$ as a superset of $\Fp$ with one extra element $\infty$. We define the function $\binv:\P^1(\Fp)\to \P^1(\Fp)$ by $\binv(x)=1/x$ if $x\ne 0,\infty$, and $\binv(0)=\infty$ and $\binv(\infty)=0$. 

If $P$ is a symmetric matrix, we let $\lambda_i(P)$ denote the $i$th largest eigenvalue of $P$.

\section{Lower bound}
\label{sec: lower bound}
In this section, we prove Theorem \ref{thm: lower bound}, which follows from entropy considerations. For a discrete probability measure $\mu$, we let 
\begin{equation*}
    H(\mu)=\sum -\mu(x)\log \mu(x)
\end{equation*}
denote the \emph{entropy} of $\mu$, with the convention that $0\log 0=0$. If $X$ is a random variable, we let $H(X)$ denote the entropy of the law of $X$. We need two basic properties of entropy. The first is that it is subadditive, in the sense that if $X$ and $Y$ are independent, then $H(X+Y)\leq H(X)+H(Y)$. The second is that for any function $f$, $H(f(X))\leq H(X)$, with equality if $f$ is bijective.

The lower bound follows by noting that the entropy of the random walk at time $n$ is at most $nH(\mu)$, which is too low to be close to uniform if $n$ is too small. This idea is certainly not new, and similar arguments have appeared before, see \cite{EV,A83, BCS19} for example. Our argument follows \cite{A83}, although we fill in some details and fix a minor error.

The following lemma shows that measures close to uniform in total variation must have large entropy. It was stated without the $\log 2$ term in \cite{A83}, but this cannot be correct, as can be seen if $\mu$ is uniform and $\nu$ is concentrated at a single point. It can be seen as a special case of Theorem 1 of \cite{A07} (in fact, it is established as part of the proof).
\begin{lemma}
\label{lem: entropy}
Let $X$ be a finite set of size $n$, and let $\pi$ denote the uniform measure on $X$. Let $\nu$ be any probability measure on $X$, and let $\delta=\|\nu-\pi\|_{TV}$. Then
\begin{equation*}
\begin{split}
    |H(\mu)-H(\nu)|&\leq \delta \log(n-1)-\delta\log \delta-(1-\delta)\log(1-\delta)
    \\&\leq \delta\log n+\log(2).
\end{split}
\end{equation*}
\end{lemma}
\begin{proof}
The first inequality is Equation 11 of \cite{A07} and the second inequality is clear.
\end{proof}

\begin{proof}[Proof of Theorem \ref{thm: lower bound}]
Because $\iota$ is a bijection, $H(X)=H(\iota (X))$. Furthermore, if $\mu_p$ denotes the distribution given by reducing $\mu$ mod $p$, then $H(\mu_p)\leq H(\mu)$. Then by subadditivity of entropy, $H(K^n(x,\cdot))\leq n H(\mu)$. Thus, Lemma \ref{lem: entropy} gives
\begin{equation*}
    \|K^n(x,\cdot)-\pi\|_{TV}\geq 1-\frac{nH(\mu)+\log 2}{\log p}.
\end{equation*}
\end{proof}

\begin{remark}
\label{rmk: necc of finite entropy}
We note that the finite entropy assumption cannot be removed entirely. Fix any constant $\delta \in (0,1)$. Consider the measure on $\Z$ given by
\begin{equation*}
    \mu(i)\propto\frac{1}{i (\log i)^{2-\delta}}.
\end{equation*}
There is some constant $c=c(\delta)>0$ such that for all $p$ and $i\le p$, 
\begin{equation*}
    \sum_{k\in\Z}\mu(i+kp)\geq \sum_{k\geq p}\frac{1}{kp(\log(kp))^{2-\delta}}\geq \frac{c}{p(\log p)^{1-\delta}}.
\end{equation*}
Then on $\Fp$, the walk defined by \eqref{eq: walk defn} satisfies 
\begin{equation*}
    \P(X_{n+1}=x|X_n=y)\geq \frac{c}{p(\log p)^{1-\delta}}
\end{equation*}
for all $x$ and $y$, and so we can couple two copies of the walk, $X_n$ and $X_n'$, so that independent of the initial states,
\begin{equation*}
    \P(X_{n}\neq X'_{n})\leq \left(1-\frac{c}{(\log p)^{1-\delta}}\right)^n.
\end{equation*}
Since the distance to stationarity can be controlled by this probability (see \cite[Corollary 5.5]{LP17}), this implies that
\begin{equation*}
    \|K^n(x,\cdot)-\pi\|_{TV}\leq \left(1-\frac{c}{(\log p)^{1-\delta}}\right)^n,
\end{equation*}
and so $t_{mix}(\ee)=O_\ee((\log p)^{1-\delta})$, which is much smaller than $\log p$. 

For positive integers $k$, define $\log^{(k)}$ iteratively by $\log^{(1)}(x)=\log x$ and $\log^{(k)}(x) = \log (\log ^{(k-1)}(x))$ for $k\ge 2$. The measure
\[
    \mu(i)\propto\frac{1}{i (\log i) (\log^{(2)} i)\cdots (\log^{(k-1)} i) (\log^{(k)} i)^{2}}
\]
gives a probability distribution on $\Z$ with infinite entropy for which the mixing time $t_{mix}(\ee) = O_{\ee}(\log^{(k)} p)$.
\end{remark}

\section{Comparison theory}
\label{sec: comparison}
Comparison theory for Markov chains was introduced by Diaconis and Saloff-Coste in \cite{DSC93a, DSC93b}. This theory allows the spectral gaps (and also log-Sobolev constants) of different Markov chains on the same state space to be compared. Unfortunately, since we wish to compare a Markov chain on $\Fp$ with one on $\P^1(\Fp)$, this theory does not immediately apply. Smith extended these ideas in \cite{S15} to the case of random walks on state spaces $X_0\subseteq X$.

We first recall the relationship between the Dirichlet form and the spectral gap, and then explain the comparison theory developed by Smith for Markov chains on different state spaces. This is used to compare the spectral gaps of the random walks on $\P^1(\Fp)$ and $\Fp$. For a more thorough treatment and proofs, we refer the reader to \cite{S15}.

\subsection{The Dirichlet form}
Recall the following standard notions. Let $P$ be a reversible Markov chain on a finite space $X$ with stationary distribution $\pi$. Define the function
\begin{equation*}
    V_\pi(f)=\frac{1}{2}\sum_{x,y\in X}|f(x)-f(y)|^2\pi(x)\pi(y) 
\end{equation*}
and the \emph{Dirichlet form}
\begin{equation*}
    \mathcal{E}_P(f,f)=\frac{1}{2}\sum _{x,y\in X}|f(x)-f(y)|^2P(x,y)\pi(x)
\end{equation*}
for $f:X\to\R$. Then the spectral gap can be computed by
\begin{equation}
\label{eq: spec gap var}
    1-\lambda_2(P)=\inf_{\substack{f:X\to\R\\f\text{ not constant}}}\frac{\mathcal{E}_P(f,f)}{V_\pi(f)}.
\end{equation}

\subsection{Comparison theory for Markov chains on different state spaces}
Let $X_0\subseteq X$, and let $P_0$ be a Markov chain on $X_0$ and $P$ be a Markov chain on $X$, with stationary distributions $\pi_0$ and $\pi$ respectively. Assume that both are ergodic and reversible. The following results of Smith compare $V_{\pi_0}$ and $\mathcal{E}_{P_0}$ with $V_\pi$ and $\mathcal{E}_P$.

\begin{lemma}[\hspace{1sp}{\cite[Lemma 2]{S15}}]
Let $f_0:X_0\to\R$ and let $f:X\to\R$ be any extension of $f_0$. Then
\begin{equation*}
    V_{\pi_0}(f_0)\leq CV_\pi(f),
\end{equation*}
where
\begin{equation*}
    C=\sup_{x\in X_0}\frac{\pi_0(x)}{\pi(x)}.
\end{equation*}
\end{lemma}

The analogous comparison result for the Dirichlet form requires further setup. For each $x\in X$, fix a probability measure $Q_x$ on $X_0$ such that $Q_x=\delta_x$ if $x\in X_0$. For any $f_0:X_0\to \R$, this defines an extension $f:X\to \R$ by
\begin{equation*}
    f(x)=\sum_{y\in X_0}Q_x(y)f(y).
\end{equation*}
Next, choose a coupling of $Q_x$ and $Q_y$ for all $x,y\in X$ for which $P(x,y)>0$ (note this is a symmetric relation). Denote these couplings by $Q_{x,y}$.

A \emph{path} in $X_0$ from $x$ to $y$ is a sequence of $x_i\in X_0$ for $i=0,\dotsc, k$ for which $x_0=x$, $x_k=y$ and $P_0(x_i,x_{i+1})>0$ for all $i$. For a path $\gamma$ from $x$ to $y$, call $x$ the \emph{initial vertex} and $y$ the \emph{final vertex}, and denote them by $i(\gamma)$ and $o(\gamma)$ respectively. Let $|\gamma|$ denote the length of the path.

A \emph{flow} on $X_0$ is a function $F$ on the set of paths in $X_0$ whose restriction to paths from $x$ to $y$ gives a probability measure for all $x$, $y$. We require a choice of flow on $X_0$. Actually it suffices to define the flow only for paths from $a$ to $b$ such that there exists $x,y \in X$ with $Q_{x,y}(a,b)>0$. For our purposes, it will suffice to choose the flow such that when restricted to paths from $a$ to $b$, it concentrates on a single path.

The following theorem compares $\mathcal{E}_{P_0}$ with $\mathcal{E}_P$ in terms of the chosen data.
\begin{theorem}[\hspace{1sp}{\cite[Theorem 4]{S15}}]
\label{thm: comparison}
Consider the setup defined above. Then for $f_0:X_0\to\R$ and $f$ the extension to $X$ defined by the measures $Q_x$, we have
\begin{equation*}
    \mathcal{E}_P(f,f)\leq \mathcal{A}\mathcal{E}_{P_0}(f_0,f_0),
\end{equation*}
where
\begin{equation*}
\begin{split}
    \mathcal{A}=&\sup_{P_0(x,y)>0}\frac{1}{P_0(x,y)\pi_0(x)}\bigg(\sum _{\gamma\ni (x,y)}|\gamma|F(\gamma)P(i(\gamma),o(\gamma))\pi(i(\gamma))
    \\&\qquad +2\sum _{\gamma\ni (x,y)}|\gamma|F(\gamma)\sum _{b\not\in X_0}Q_b(o(\gamma))P(i(\gamma),b)\pi(i(\gamma))
    \\&\qquad +\sum _{\gamma\ni (x,y)}|\gamma|F(\gamma)\sum_{\substack{a,b\not\in X_0\\P(x,y)>0}}Q_{a,b}(i(\gamma),o(\gamma))P(a,b)\pi(a)\bigg).
\end{split}
\end{equation*}
\end{theorem}

Note that by \eqref{eq: spec gap var}, an immediate consequence of these comparison results is that
\begin{equation*}
    1-\lambda_2(P_0)\geq \frac{1}{C\mathcal{A}}(1-\lambda_2(P)).
\end{equation*}

\section{Upper bound}
\label{sec: proof}

In this section, we prove Theorem \ref{thm: upper bound}. The proof of the upper bound consists of three main steps. First, we control the mixing time by the spectral gap of a symmetrized random walk (Corollary \ref{cor: red to sym}). Next, we relate the spectral gap of the symmetrized walk on $\Fp$ with a walk on $\P^1(\Fp)$ (Proposition \ref{prop: comp}). Finally, we show that the spectral gap of the walk on $\P^1(\Fp)$ is of constant order (Proposition \ref{prop: spectral gap for P1}).

\begin{proof}[Proof of Theorem \ref{thm: upper bound}]
Let $P$ denote the Markov matrix encoding the step $X\mapsto X+\ee$ and let $\Pi$ denote the Markov matrix encoding $X\mapsto \inv(X)$. Note that $P^T$ encodes the step $X\mapsto X-\ee$. By Corollary \ref{cor: red to sym}, it suffices to show that
\begin{equation*}
    \lambda_2(P^T\Pi P^T P\Pi P)\leq 1-c
\end{equation*}
for some constant $c>0$, independent of $p$. 

Let $a_1,a_2\in\Z$ be two distinct elements in the support of $\mu$ and write $b=a_1-a_2$. Then in the Markov chain given by the transition matrix $P^T\Pi P^T P\Pi P$, there is $u>0$ depending only on $\mu$ such that we transition from $x$ to one of $x+b$, $x-b$, $\inv(\inv(x+a_1)+b)-a_1$ and $\inv(\inv(x+a_1)-b)-a_1$ with probability at least $u$. For example, note that $x+b=\iota(\iota(x+a_1)+a_1-a_1)-a_2$ and so there is a positive probability of moving from $x$ to $x+b$.

Then we can write $P^T\Pi P^T P\Pi P=uL_0+(1-u)L'$ where $L_0$ is the transition matrix for the random walk going from $x$ to one of
$x+b$, $x-b$, $\inv(\inv(x+a_1)+b)-a_1$ and $\inv(\inv(x+a_1)-b)-a_1$ with equal probability, and $L'$ a symmetric stochastic matrix. We have
\begin{equation*}
    \lambda_2(P^T\Pi P^T P\Pi P)\leq u\lambda_2(L_0)+1-u,
\end{equation*}
and so it suffices to show that $\lambda_2(L_0)\leq 1-c$ for some constant $c>0$.

Let $\binv:\P^1(\Fp)\to \P^1(\Fp)$ be the function defined by $\binv(x)=1/x$ if $x\ne 0,\infty$, and $\binv(0)=\infty$, and $\binv(\infty)=0$. Let $L$ denote the transition matrix for the random walk on $\P^1(\Fp)$ going from $x$ to one of $x+b$, $x-b$, $\binv(\binv(x+a_1)+b)-a_1$ and $\binv(\binv(x+a_1)-b)-a_1$ with equal probability. By Proposition \ref{prop: comp}, we can instead show
\begin{equation*}
    \lambda_2(L)\leq 1-c
\end{equation*}
for the walk defined by $L$ instead of the walk defined by $L_0$, and this is exactly given by Proposition \ref{prop: spectral gap for P1}.
\end{proof}

\subsection{Reduction to the symmetric walk}
\label{sec: red to sym}
Recall that we are interested in studying the walk $X_{n+1}=\inv(X_n)+\ee_{n+1}$ on $\Fp$, where $\ee_n$ are independent and identically distributed according to $\mu$. This random walk is non-reversible, so the first step is to relate it to a suitable symmetrization.

Let $\Pi$ denote the transition matrix for the (deterministic) walk $X\mapsto \inv(X)$ on $\Fp$ and let $P$ denote the walk $X\mapsto X+\ee$. Note that $P^T$ is also a Markov matrix and encodes the walk $X\mapsto X-\ee$. Let $\lambda_2$ be the second largest eigenvalue of $P^T\Pi P^T P\Pi P$. Note that all eigenvalues of $P^T\Pi P^T P\Pi P$ are real and non-negative.

The following result from \cite{CD20} relates the original walk $K=P\Pi$ to the symmetrized walk $P^T\Pi P^T P\Pi P$.

\begin{lemma}[\hspace{1sp}{\cite[Corollary 5.4]{CD20}}]
\label{lem: symmetrized bound}
Let $k\geq 2$ and let $x\in\R^p$ such that $\sum x_i=0$. Then
\begin{equation*}
    \|(K^T)^kx\|_2\leq \lambda_2^{(k-2)/4}\|x\|_2.
\end{equation*}
\end{lemma}

This result means that if $P^T\Pi P^T P\Pi P$ can be shown to have a spectral gap of constant order, then the walk $K$ mixes in order $\log p$ steps. The following is an easy consequence of Lemma \ref{lem: symmetrized bound} (see Equation 5.5 in \cite{CD20}).
\begin{corollary}
\label{cor: red to sym}
Let $\pi$ denote the uniform measure on $\Fp$. Then for all $x\in\Fp$,
\begin{equation*}
    \|K^k(x,\cdot)-\pi\|_{TV}\leq \frac{\sqrt{p}}{2}\lambda_2^{(k-2)/4}.
\end{equation*}
\end{corollary}

\subsection{Comparison of the random walks on \texorpdfstring{$\P^1(\Fp)$}{P1(Fp)} and \texorpdfstring{$\Fp$}{Fp}}
We now explain how to apply the comparison results to bound the spectral gap of the walk on $\Fp$ in terms of the spectral gap of the walk on $\P^1(\Fp)$.

Recall from the proof of Theorem \ref{thm: upper bound} that we denote by $L_0$ the transition matrix of the random walk on $\Fp$ which moves from $x\in \Fp$ to one of $x+b$, $x-b$, $\inv(\inv(x+a_1)+b)-a_1$ and $\inv(\inv(x+a_1)-b)-a_1$ with equal probability. We denote by $L$ the transition matrix of the walk on $\P^1(\Fp)$ which moves from $x\in \Fp$ to one of $x+b$, $x-b$, $\binv(\binv(x+a_1)+b)-a_1$ and $\binv(\binv(x+a_1)-b)-a_1$. Here, recall that $\binv(x)=1/x$ if $x\ne 0,\infty$, and $\binv(0)=\infty$, and $\binv(\infty)=0$. The following lemma is useful in constructing the data required in Theorem \ref{thm: comparison}.

\begin{lemma}
\label{lem: graph inclusion}
For all $x,y\in \Fp$ except $x=-a_1$ and $y=-a_1$, if $L(x,y)>0$, then $L_0(x,y)>0$.
\end{lemma}
\begin{proof}
First, note that if $x\neq -a_1, b^{-1}-a_1,-b^{-1}-a_1$, then the statement is clear since in this case the functions $\inv$ and $\binv$ are identical for the transitions involved. It can be checked that if $x=b^{-1}-a_1$ or $x=-b^{-1}-a_1$, and $y \ne \infty$ and $L(x,y) > 0$, then $L_0(x,y)>0$. Similarly, we can check that if $x=-a_1$ and $L(x,y)>0$ then $L_0(x,y)>0$ unless $y=-a_1$.
\end{proof}

\begin{proposition}
\label{prop: comp}
Let $L_0$ and $L$ denote the transition matrices for the random walk on $\Fp$ and $\P^1(\Fp)$ respectively. Then there exists an absolute constant $c>0$ such that
\begin{equation*}
    1-\lambda_2(L_0)\geq c(1-\lambda_2(L)).
\end{equation*}
\end{proposition}
\begin{proof}
We begin by defining the data needed to apply Theorem \ref{thm: comparison}. Take $X_0=\Fp$ and $X=\P^1(\Fp)$. Define for $x\in \Fp$, $Q_x=\delta_x$ and $Q_\infty$ is uniform on the set of $y\in \Fp$ such that $L(\infty,y)>0$. Take all couplings to be independent, so $Q_{x,y}(x_0,y_0)=Q_x(x_0)Q_y(y_0)$. Then note that the condition that $Q_{x,y}(x_0,y_0)>0$ for some $x,y\in X$ is exactly that either $L(x_0,y_0)>0$, or $L(x_0,\infty)>0$ and $L(y_0,\infty)>0$.

By Lemma \ref{lem: graph inclusion}, if $L(x_0,y_0)>0$, then $L_0(x_0,y_0)>0$ except when $x_0=-a_1$ and $y_0=-a_1$. Thus, if $L(x_0,y_0)>0$, we pick our flow $F$ to be concentrated on the path consisting of the single edge $(x_0,y_0)$, except when $x_0=y_0=-a_1$ where we take the path of length 2 given by $-a_1\mapsto b^{-1}-a_1\mapsto -a_1$.

If $L(x_0,\infty)>0$ and $L(y_0,\infty)>0$, then there must be a path of length $2$ connecting $x_0$ to $y_0$. This is because $x_0,y_0\in \{b^{-1}-a_1,-b^{-1}-a_1\}$ and so we can take our flow to be concentrated on the path $x_0\mapsto -a_1\mapsto y_0$. Thus, our flow is concentrated entirely on paths of lengths $1$ and $2$.

We can now compute the constants $C$ and $\mathcal{A}$. It's easy to see that $C\leq 2$. To bound $\mathcal{A}$, note that for each $(x,y)$ for which $L_0(x,y)>0$, there is at most one path $\gamma$ of length $1$ on which $F(\gamma)\neq 0$ containing $(x,y)$, namely $\gamma=(x,y)$, and there is a bounded number of paths $\gamma$ of length $2$ containing $(x,y)$ for which $F(\gamma)>0$. Thus, for any $(x,y)$ for which $L_0(x,y)>0$, there are a bounded number of terms in the summation in the definition of $\mathcal{A}$, each of bounded size. Hence, $\mathcal{A}$ is 
bounded above by a constant. This implies that
\begin{equation*}
    1-\lambda_2(L_0)\geq c(1-\lambda_2(L))
\end{equation*}
for some absolute constant $c>0$.
\end{proof}
\begin{remark}
For a reversible Markov chain $P$ with stationary distribution $\pi$, let
\begin{equation*}
    \phi(P)=\min_{\pi(S)\leq 1/2}\frac{\sum_{x\in S, y\in S^c}\pi(x)P(x,y)}{\pi(S)}
\end{equation*}
denote the \emph{bottleneck ratio} (sometimes also called the \emph{Cheeger constant} or \emph{conductance}). We note that a constant order spectral gap for $L_0$ could also be derived using Cheeger's inequality (see \cite[Theorem 13.10]{LP17} for example), which states that
\begin{equation*}
    \frac{\phi(P)^2}{2}\leq 1-\lambda_2(P)\leq 2\phi(P).
\end{equation*}
The bottleneck ratios of $L_0$ and $L$ can be easily compared, which would give a uniform lower bound for $1-\lambda_2(L_0)$ using the one for $1-\lambda_2(L)$. We prefer to use comparison theory because the argument can be more easily generalized. In particular, it is more robust and would give a sharper result if $L_0$ did not have a constant order spectral gap.
\end{remark}

\subsection{Spectral gap for random walk on \texorpdfstring{$\P^1(\Fp)$}{P1(Fp)}}
In this section, we prove that the random walk on $\P^1(\Fp)$ has a constant order spectral gap. 

Recall that $\SL_2(\Fp)$ has a transitive action on $\P^1(\Fp)$, viewed as lines in $\Fp^2$. This is through linear fractional transformations, and may be formally defined by
\begin{equation}\label{linear trans}
\begin{pmatrix}
a & b  \\
c & d \\
\end{pmatrix} \cdot x = \begin{cases} \frac{a x + b  }{cx +d}&\text{if $cx+d \neq 0$ and $x\in \Fp$} \\
\infty &\text{if $cx+d = 0$ and $x \in \Fp$}\\
\frac{a}{c}&\text{if $x  = \infty$ and $c \neq 0$}\\
\infty &\text{if $x = \infty$ and $c = 0$}.  
\end{cases}
\end{equation}

The random walk on $\P^1(\Fp)$ given by moving from $x$ to one of $x+b$, $x-b$, $\binv(\binv(x+a_1)+b)-a_1$ and $\binv(\binv(x+a_1)-b)-a_1$ uniformly at random can be viewed as the quotient of a random walk on $\SL_2(\Fp)$. The following result of Bourgain and Gamburd gives a constant order spectral gap for random walks on $\SL_2(\Fp)$.
\begin{theorem}[\hspace{1sp}{\cite[Theorem 1]{BG}}]
\label{thm: SL2 spec gap}
Let $S\subseteq \SL_2(\Z)$ be a symmetric set which generates a non-elementary subgroup of $\SL_2(\Z)$. Let $S_p$ denote the set of generators mod $p$ for a prime $p$. Let $P$ denote the transition matrix for the random walk on $\SL_2(\Fp)$ defined by
\begin{equation*}
    X_{n+1}=X_n\ee_{n+1},
\end{equation*}
where the $\ee_i$ are independent and uniformly distributed on $S_p$. Then for all primes $p$ large enough,
\begin{equation*}
    \lambda_2(P)\leq 1-c
\end{equation*}
for some constant $c>0$ independent of $p$.
\end{theorem}

To show that this applies to the random walk we wish to study, we need the following lemma.
\begin{lemma}
\label{lem: Zariski dense}
Let
    \begin{equation*}
    S=\left\{\begin{pmatrix}
1 & b  \\
0 & 1\\
\end{pmatrix} , \begin{pmatrix}
1-a_1 b & -a_1^2 b  \\
b & a_1 b+1 \\
\end{pmatrix},\begin{pmatrix}
1 & -b  \\
0 & 1\\
\end{pmatrix} , \begin{pmatrix}
1 + a_1 b & a_1^2 b  \\
-b & 1 - a_1b  \\
\end{pmatrix}\right\}.
\end{equation*}
Then $S$ generates a non-elementary subgroup of $\SL_2(\Z)$.
\end{lemma}
\begin{proof}
By Theorem 2.5 of \cite{R10}, non-elementary subgroups are the same as Zariski-dense subgroups in $\SL_2(\Z)$. By a result of Weigel \cite{W96} (see also Theorem 2.2 of \cite{R10}), it suffices to show that after reducing mod $p$ for some $p\geq 5$, $S_p$ generates $\SL_2(\Fp)$. 

If $b \ne 0 \pmod p$, then the matrices $\begin{pmatrix}
1 & b  \\
0 & 1\\
\end{pmatrix}$ and $\begin{pmatrix}
1 & -b  \\
0 & 1\\
\end{pmatrix}$ generate all matrices of the form $\begin{pmatrix}
1 & t  \\
0 & 1\\
\end{pmatrix}$ for $t\in \Fp$. 

Using $\begin{pmatrix}
1 & t  \\
0 & 1\\
\end{pmatrix}$ and $\begin{pmatrix}
1-a_1 b & -a_1^2 b  \\
b & a_1 b+1 \\
\end{pmatrix}$, we can generate all matrices of the form $$\begin{pmatrix}
1 & t'  \\
0 & 1\\
\end{pmatrix} \cdot \begin{pmatrix}
1-a_1 b & -a_1^2 b  \\
b & a_1 b+1 \\
\end{pmatrix} \cdot 
\begin{pmatrix}
1 & t  \\
0 & 1\\
\end{pmatrix}.$$ In particular, we can generate all matrices in the subset $X_b$ of $\SL_2(\Fp)$ of matrices whose lower left corner equal to $b$. 

Similarly, $S_p$ generates matrices in the subset $X_{-b}$ of $\SL_2(\Fp)$ of matrices with lower left corner equal to $-b$, using $\begin{pmatrix}
1 + a_1 b & a_1^2 b  \\
-b & 1 - a_1b  \\
\end{pmatrix}$ instead of $\begin{pmatrix}
1-a_1 b & -a_1^2 b  \\
b & a_1 b+1 \\
\end{pmatrix}$. 

Finally, we can easily check that $X_b$ and $X_{-b}$ generate $\SL_2(\Fp)$. Thus, $S_p$ generates $\SL_2(\Fp)$ for all $p>|b|$, and hence, $S$ generates a non-elementary subgroup of $\SL_2(\Z)$. 
\end{proof}

Since the random walk on $\P^1(\Fp)$ is a quotient of a random walk on $\SL_2(\Fp)$, we can obtain the desired bound on the spectral gap.
\begin{proposition}
\label{prop: spectral gap for P1}
Let $L$ denote the transition matrix for the random walk on $\P^1(\Fp)$ which moves from $x\in \Fp$ to one of $x+b$, $x-b$, $\binv(\binv(x+a_1)+b)-a_1$ and $\binv(\binv(x+a_1)-b)-a_1$ uniformly at random. Then
\begin{equation*}
    \lambda_2(L)\leq 1-c
\end{equation*}
for some constant $c>0$ independent of $p$.
\end{proposition}
\begin{proof}
Note that the formulas given in \ref{linear trans} actually define a $\GL_2(\Fp)$ action. It's easy to see that $\iota$ and addition by $a\in\Fp$ are both linear frational transformations, represented by the matrices
\begin{equation*}
    \begin{pmatrix}
1 & a  \\
0 & 1\\
\end{pmatrix},
\begin{pmatrix}
0 & 1  \\
1 & 0\\
\end{pmatrix}
\end{equation*}
in $\GL_2(\Fp)$. Then we may write the function $x\mapsto x+b$ as a product of these matrices, and similarly for $x-b$, $\binv(\binv(x+a_1)+b)-a_1$ and $\binv(\binv(x+a_1)-b)-a_1$. This allows us to view the walk defined by $L$ as the quotient of the walk on $\SL_2(\Fp)$ generated by the set
\begin{equation*}
    S=\left\{\begin{pmatrix}
1 & b  \\
0 & 1\\
\end{pmatrix} , \begin{pmatrix}
1-a_1 b & -a_1^2 b  \\
b & a_1 b+1 \\
\end{pmatrix},\begin{pmatrix}
1 & -b  \\
0 & 1\\
\end{pmatrix} , \begin{pmatrix}
1 + a_1 b & a_1^2 b  \\
-b & 1 - a_1b  \\
\end{pmatrix}\right\},
\end{equation*}
with transition matrix $\widetilde{L}$. This means that the spectrum of $L$ is contained in the spectrum of $\widetilde{L}$, and in particular $\lambda_2(L)\leq \lambda_2(\widetilde{L})$. But since $S$ generates a non-elementary subgroup of $\SL_2(\Z)$ by Lemma \ref{lem: Zariski dense}, Theorem \ref{thm: SL2 spec gap} implies that
\begin{equation*}
    \lambda_2(\widetilde{L})\leq 1-c
\end{equation*}
for some constant $c>0$ independent of $p$, giving us the desired bound. 
\end{proof}

\section{Points on modular hyperbolas}
\label{sec: application}
In this section, we prove the following theorem stating that for all subsets $A\subseteq \Fp$ of size at most $p/2$, the sets $A$ and $\iota(A)$ cannot both be close to intervals. As a corollary, we obtain bounds on the number of solutions to the congruence $xy=1\pmod{p}$ lying in $I\times J\subseteq \Fp^2$ for $I$ and $J$ intervals of the same length. 

\begin{theorem}
\label{thm: application}
Let $I$ and $J$ be two intervals in $\Fp$, each of length $m\leq p/2$, and let $A\subseteq \Fp$ with $A\subseteq I$ and $\iota(A)\subseteq J$. There exists an absolute constant $\delta > 0$ such that for all $p$ and $m$ sufficiently large, $|A|\leq (1-\delta)m$. 
\end{theorem}
\begin{proof}
Throughout the proof, we use the notation $[-k,k]$ to denote the set of integers $\{-k,-k+1,\dots,k-1,k\}$. Let $Q=P^T\Pi P^TP\Pi P$, where $P$ denotes the transition matrix for the random walk on $\Fp$ generated by the uniform measure on $[-1,1]$ and $\Pi$ encodes the bijection $\iota$. Then $Q$ is the transition matrix of a reversible Markov chain on $\Fp$, and from the proof of Theorem \ref{thm: upper bound} it has a constant order spectral gap, i.e. for large enough $p$, $1-\lambda_2(Q)\geq \gamma$ for some absolute constant $\gamma>0$.

Assume that there exists intervals $I$, $J$ such that $A\subseteq I$, $\iota(A)\subseteq J$, and $|A|\geq (1-\delta)m$. We will show using Cheeger's inequality that this implies the spectral gap cannot be of constant order, which gives a contradiction.

Observe that $Q(x,y) \ge c > 0$ for some constant $c$ independent of $p$. Since the stationary distribution is uniform, the bottleneck ratio of the chain $X_n$ is given by 
\[
\Phi = \min_{S \subseteq \Fp, |S| \le p/2} \frac{\sum_{x\in S,y\notin S} Q(x,y)}{|S|}.
\]
By Cheeger's inequality, we have $\Phi$ is bounded below by a positive constant $\gamma'$ depending only on $\gamma$. Thus,
\[
\frac{\sum_{x\in A,y\notin A} Q(x,y)}{|A|} \ge \gamma'. 
\]

For $x\in A$, we consider $y\in \Fp$ for which $Q(x,y)>0$. Then \begin{equation*}
    y\in \iota(\iota(A+[-1,1])+[-2,2])+[-1,1].
\end{equation*}
Note that $\iota(A+[-1,1])\subseteq \iota(I+[-1,1])$, and
\begin{align*}
    \iota(I + [-1,1]) &= \iota(A \cup ((I + \{-1,+1\})\setminus A)) \\
    &= \iota(A) \cup \iota((I + \{-1,+1\})\setminus A).
\end{align*}
Let $S_1 = \iota((I + \{-1,+1\})\setminus A)$. Then $\iota(A+[-1,1])\subseteq \iota(A)\cup S_1$, and $|S_1|\leq 2+\delta m$. 

Similarly, since $\iota(A)\cup S_1 \subseteq J\cup S_1$, we have
\begin{align*}
    \iota((\iota(A)\cup S_1) + [-2,2]) &\subseteq \iota(J+[-2,2]) \cup \iota(S_1 + [-2,2]) \\
    &\subseteq A \cup \iota((J+[-2,2])\setminus \iota(A)) \cup \iota(S_1 + [-2,2]).
\end{align*}
Here, $|\iota((J+[-2,2])\setminus \iota(A))| \le 4 + \delta m$, and $|\iota(S_1 + [-2,2])| \le 5|S_1|$. Letting 
\begin{equation*}
    S_2=\iota((J+[-2,2])\setminus \iota(A)) \cup \iota(S_1 + [-2,2]),
\end{equation*}
we have
\begin{equation*}
    \iota(\iota(A+[-1,1])+[-2,2])\subseteq A\cup S_2,
\end{equation*}
where $|S_2| \le 14+6\delta m$. 

Finally, we have $y \in (A \cup S_2) + [-1,1]$, and 
\begin{align*}
    (A \cup S_2) + [-1,1]&\subseteq (I \cup S_2) + [-1,1] \\ 
    &\subseteq (I+[-1,1]) \cup (S_2 + [-1,1])\\
    &\subseteq A \cup ((I+[-1,1]) \setminus A) \cup (S_2 + [-1,1]).
\end{align*}
Let $S_3 = ((I+[-1,1]) \setminus A) \cup (S_2 + [-1,1])$, and note that $|S_3| \le 44+19\delta m$. 

Thus, if $x\in A$, $y\notin A$, and $Q(x,y)>0$, then we must have $y\in S_3$. Hence,
\[
\frac{\sum_{x\in A,y\notin A} Q(x,y)}{|A|} \le \frac{\sum_{x\in A,y\in S_3} Q(x,y)}{|A|} \le \frac{44+19\delta m}{(1-\delta)m} \le \frac{20\delta}{1-\delta}, 
\]where we assumed that $m$ is sufficiently large (in $\delta$), and used that 
\[
\sum_{x\in A,y\in S_3} Q(x,y) \le \sum_{x\in \Fp,y\in S_3} Q(y,x) = |S_3|.
\]
For $\delta$ sufficiently small, $20\delta/(1-\delta)<\gamma'$, and we have a contradiction. 
\end{proof}

The following corollary follows easily from Theorem \ref{thm: application} by taking $A=\{x\in I\mid \iota(x)\in J\}$, which away from $0$ counts the solutions of interest.
\begin{corollary}
Let $I$ and $J$ be two intervals in $\Fp$ of the same length $m\leq p/2$. There exists an absolute constant $\delta>0$ such that for sufficiently large $p$ and $m$,
\begin{equation*}
    \left|\{(x,y)\in I\times J\mid xy\equiv 1\pmod{p}\}\right|\leq (1-\delta)m.
\end{equation*}
\end{corollary}

\section*{Acknowledgements}
The authors thank Sourav Chatterjee, Persi Diaconis, Jacob Fox, Sean Eberhard, 
Ilya Shkredov, Kannan Soundararajan, P\'eter Varj\'u, Thuy Duong Vuong, and Yuval Wigderson for their help and comments on earlier drafts. We are greatly indebted to P\'eter Varj\'u for pointing out the many useful references to the study of spectral gap of Cayley graphs of $\SL_2(\Fp)$, in particular, \cite{BG}.

\bibliographystyle{plain}
\bibliography{mix}{}
\end{document}